[1994/12/01]
\documentclass{amsart}
\usepackage{amssymb}
\usepackage[all]{xy}
\usepackage[dvips]{graphicx}
\chardef\bslash=`\\ 

\newtheorem{theorem}{Theorem}[section]
\newtheorem*{SCtheorem}{Theorem \ref{fctsc}}

\newtheorem*{GSPtheorem}{Theorem \ref{GSP}}
\newtheorem*{APtheorem}{Theorem \ref{dist}}

\newtheorem{corollary}[theorem]{Corollary}
\newtheorem{lemma}[theorem]{Lemma}

\newtheorem{proposition}[theorem]{Proposition}

\theoremstyle{definition}
\newtheorem{definition}{Definition}[section]

\newtheorem*{ack}{Acknowledgement}

\theoremstyle{remark}

\DeclareMathOperator{\real}{Re}
\DeclareMathOperator{\Img}{Im}

\DeclareMathOperator{\ord}{ord}
\DeclareMathOperator{\st}{st}
\DeclareMathOperator{\supp}{supp}
\DeclareMathOperator{\Div}{Div}

\newcommand{\eval}[2][\right]{\relax
  \ifx#1\right\relax \left.\fi#2#1\rvert}

\begin{document}
\title{A factorization of a super-conformal map}
\author{Katsuhiro Moriya}
\address{Division of Mathematics, Faculty of Pure and Applied Sciences, 
University of Tsukuba\\
 1-1-1 Tennodai, Tsukuba, Ibaraki 305-8571\\Japan}
\email{moriya@math.tsukuba.ac.jp} 
\issueinfo{}{}{}{}
\begin{abstract}
A super-conformal map and a minimal surface are factored into a product of two maps by modeling the Euclidean four-space and the complex Euclidean plane on the set of all quaternions. 
One of these two maps is a holomorphic map or a meromorphic map. 
These conformal maps adopt properties of a holomorphic function or a meromorphic function. 
Analogs of the Liouville theorem, the Schwarz lemma, the Schwarz-Pick theorem, the Weierstrass factorization theorem, the Abel-Jacobi theorem, and a relation between zeros of a minimal surface and branch points of a super-conformal map are obtained. 
\end{abstract}
\maketitle
\tableofcontents
\section{Introduction}
Pedit and Pinkall considered 
a conformal map from a Riemann surface to the Euclidean space of dimension four to be an analog of a holomorphic function or a meromorphic function in \cite{PP98}, 
by modeling $\mathbb{R}^4$ on the set of all quaternions $\mathbb{H}$. 
Let $M$ be a Riemann surface with complex structure $J$. 
Given a conformal map from $M$ to $\mathbb{R}^4$, 
there exists a complex structure of $\mathbb{R}^4$, parametrized by $M$, such that 
the conformal map is holomorphic about this complex structure at each point of $M$ (see Figure \ref{analogy}). 

   \begin{table}[h]
   \begin{center}
   \begin{tabular}{l|l}
   map&equation\\
   \hline
A conformal map $f\colon M\to\mathbb{H}$&$df\circ J=N\,df$,\\
&$N\colon M\to\Img\mathbb{H}$, $N^2=-1$\\
A holomorphic function $h\colon M\to\mathbb{C}$&$dh\circ J=i\,dh$
\end{tabular}
\end{center}
   \caption{A conformal map and a holomorphic function.}
   \label{analogy}
   \end{table}

The motivation for making this interpretation of a conformal map is to obtain its global properties. 
After a long history, the theory of meromorphic functions has been successful in obtaining global properties of meromorphic functions. 
Compared with the theory of meromorphic functions, the theory of conformal maps seems 
to be insufficiently developed in obtaining global properties. 

The above interpretation raises a problem whether a conformal map adopts to a property of a meromorphic function.
In \cite{PP98}, the order of a zero of a conformal map and 
the degree of a conformal map are defined (Theorem 3.2, Definition 3.2). 
The Riemann-Roch theorem for conformal maps is proved (Section 4). 
Quaternionic holomorphic curves in the 
quaternionic projective space in \cite{FLPP01} shows various properties of meromorphic functions. 

This important achievement mainly arises for a compact Riemann surface without boundary. 
However, studies on open Riemann surfaces are still in their infancy.
We will study whether a conformal map adopts properties of a meromorphic function on an open Riemann surface.

Factoring a conformal map into a product of maps is an effective way to attack this problem. 
The example, in Example of \cite{PP98}, implies that a conformal map can be factored into a product of two conformal maps. 
The topic in \cite{Moriya08} is considered as
whether a Lagrangian conformal map becomes a product of 
two Lagrangian conformal maps. 

Using a complex structure, we identify the quaternionic vector space $\mathbb{H}$ with 
a complex vector space $\mathbb{C}^2$. 
If one of the factors of a conformal map is a meromorphic map into $\mathbb{C}^2$, 
then the conformal map adopts a properties of a meromorphic function, such as zeros and poles,  through the meromorphic factor.
If the remaining factor reflects properties of a conformal map, this factorization would be a useful factorization in investigating a conformal map. 

In this paper, we provide this type of factorization of 
a super-conformal map (Theorem \ref{fctsc}) and a minimal surface (Theorem \ref{fctmin}). 

A super-conformal map is a conformal map whose curvature ellipse is a circle at each point (\cite{BFLPP02}).  
A holomorphic map and an anti-holomorphic map from $M$ to $S^2\cong\mathbb{C}P^1$ 
are super-conformal maps (see Lemma \ref{mersc}). 
The curvature ellipse is a central topic of the study of surfaces in $\mathbb{R}^4$. 
There is comprehensive explanation for classical results in \cite{Wong46}, \cite{Wong52}, and \cite{Friedrich97}. 
A super-conformal map is also called a Bor\r{u}vka's surface after Bor\r{u}vka's study (\cite{Boruvka28})
or a Wintgen ideal surface (\cite{PV08}) because the equality holds in Wintgen's inequality (\cite{Wintgen79}).  
A superminimal surface in $\mathbb{R}^4$ in \cite{Friedrich97} is a minimal super-conformal map. 

Rouxel \cite{Rouxel89} showed that a conformal transform of a super-conformal map is a super-conformal map. 
Castro \cite{Castro04} showed that the Whitney sphere is the only Lagrangian super-conformal map from 
a closed Riemann surface.  
Chen \cite{ChenBY10}  classified all super-conformal maps such that 
the absolute value of the Gaussian curvature is equal to that of the normal curvature at each point. 
Friedrich \cite{Friedrich97} described superminimal surfaces in terms of the twistor space. 
A super-conformal map in $\mathbb{R}^4$ is a stereographic projection of $S^4$ composed with the twistor projection of 
a holomorphic map from $M$ to $\mathbb{C}P^3$.
Thus 
a super-conformal map is a Willmore surface with vanishing Willmore energy  (see \cite{BFLPP02}). 
In \cite{Moriya09}, \cite{Moriya10} and \cite{DT09}, it is shown that 
a holomorphic null curve is associated with a super-conformal map.

In \cite{BFLPP02}, two maps from $M$ to $S^2\cong\mathbb{C}P^1$ are associated with a conformal map from $M$ to $\mathbb{H}$. These maps are called the left normal and the right normal of a conformal map. 
A super-conformal map has an anti-holomorphic left normal or an anti-holomorphic right normal. 
A minimal surface has a holomorphic left normal and a holomorphic right normal. 

We contemplate a super-conformal map with anti-holomorphic left normal. 
If $M$ is a closed Riemann surface and $N\colon M\to S^2\cong \mathbb{C}P^1$ is a non-constant holomorphic map or a non-constant anti-holomorphic map, 
then $N$ is surjective. 
We consider the case where 
$M$ is an open Riemann surface and 
the left normal is a non-constant holomorphic map or a non-constant anti-holomorphic map $N\colon M\to N(M)\subsetneq S^2$. 

Let $E$ be the eigenbundle of the left regular representation of $N$ on $\mathbb{H}$ with eigenvalue $+i$. 
Then $E$ has 
a global super-conformal trivializing section 
$\psi\colon M\to\mathbb{H}$ (Lemma \ref{gscs}). 
Then we have the following factorization: 
\begin{SCtheorem}[Factorization theorem for super-conformal maps]
Let $V=M\times \mathbb{H}$ with the projection 
$\pi\colon V\to M$ be the trivial right quaternionic 
line bundle over $M$, 
$N\colon M\to N(M)\subsetneq S^2\cong\mathbb{C}P^1$ be an 
anti-holomorphic map, 
$E\subset V$ be the eigenbundle of the left regular representation of $N$ on $\mathbb{H}$ with eigenvalue $+i$ and 
$\psi$ be a global super-conformal trivializing section of $E$. 
A map $f\colon M\to\mathbb{H}$ is a super-conformal map with anti-holomorphic left normal $N$ 
if and only if 
$f=\psi(\lambda_0+\lambda_1j)$ with holomorphic functions $\lambda_0$ and $\lambda_1$ on $M$. 
\end{SCtheorem}

The factorization of super-conformal maps provides Liouville's theorem, 
the Schwarz lemma,  
the Schwarz-Pick theorem, 
the geometric version of the Schwarz-Pick theorem, the Weierstrass factorization theorem and the Abel-Jacobi theorem  for super-conformal maps (Theorem \ref{Liouville}, Theorem \ref{Schwarz}, Theorem \ref{SP}, Theorem \ref{GSP}, Theorem \ref{WFT} and Theorem \ref{Abel}). For example, we have the following:
\begin{GSPtheorem}[The geometric version of the Schwarz-Pick theorem for super-conformal maps]
Let $f\colon B^2\to B^4\subset\mathbb{H}$ be a super-conformal map with anti-holomorphic left normal   
$N\colon B^2\to S^2$. 
Then, at each point $z$ in $P^f$, 
there exists a constant $C^{z}>0$ such that $f^\ast ds^2_{B^4}\leq (C^{z})^2\,ds^2_{B^2}$. 
\end{GSPtheorem}
We use 
the geometric version of the Schwarz-Pick theorem (for holomorphic maps) 
to investigate whether a complex manifold is hyperbolic. 
For an injective super-conformal immersion from $B^2$ to $B^4$, 
we define a pseudodistance $d_f$ on $f(B^2)$ in a similar way to 
the Kobayashi pseudodistance. By the geometric version of the Schwarz-Pick theorem for super-conformal maps, 
we have a sufficient condition for $d_f$ to be a distance as follows. 
\begin{APtheorem}
Let $f\colon B^2\to B^4\subset \mathbb{H}$ be an injective super-conformal immersion with anti-holomorphic left normal 
$N\colon B^2\to N(B^2)\subsetneq S^2\cong\mathbb{C}P^1$.
Assume that $f(0)=0$.  
Let $C^z$ be a positive constant such that 
$f^\ast ds^2_{B^4}\leq (C^z)^2ds^2_{B^2}$ at $z\in B^2$. 
If there exists a constant $C>0$ such that 
$C^z\leq C$ for any $z\in B^2$,  
then $d_f$ is a distance on $f(B^2)$. 
\end{APtheorem}
The details of the theorems are explained later. 

We have a factorization theorem for minimal surfaces (Theorem \ref{fctmin}).
From this factorization, we have a relation between zeros of a minimal surface and branch points of a super-conformal map 
(Theorem \ref{zero}). 

\begin{ack}
The author would like to thank Kazumi Tsukada for helpful conversations and the referee for reading the manuscript carefully and for suggesting its improvement. 

This work is supported by JSPS KAKENHI Grant Numbers 22540064, 25400063. 
\end{ack}

\section{Conformal maps}
Throughout this paper, all manifolds and maps are assumed to be smooth.

We recall the notion of a conformal map from a Riemann surface to $\mathbb{R}^4$ (\cite{PP98}) 
and 
introduce the notion of a pole and a divisor of a conformal map. 

Let $M$ be a Riemann surface with complex structure $J^M$. 
For a one-form $\omega$ on $M$, 
we define a one-form $\ast\,\omega$ on $M$ by setting $\ast\,\omega:=\omega\circ J^M$. 
A one-form $\omega$ with values in the set of all complex numbers $\mathbb{C}$ is decomposed into the one-form of type $(1,0)$ and that of type $(0,1)$ (see Forster \cite{Forster91}). 

We model $\mathbb{R}^4$ on the set of all quaternions $\mathbb{H}$ and 
$\mathbb{R}^3$ on the set of all purely imaginary quaternions $\Img\mathbb{H}$. 
For $a\in\mathbb{H}$, we denote by $\real a$ the real part of $a$ and 
by $\Img a$ the imaginary part of $a$.  
For a quaternion $a$, we denote by $\overline{a}$ the quaternionic conjugate of $a$. 
Then, the inner product of $a$ and $b\in\mathbb{H}$ is 
$\langle a,b\rangle:=\real(\overline{a}b)=2^{-1}(\overline{a}b+\overline{b}a)$ and 
the norm of $a\in\mathbb{H}$ is $|a|:=(\overline{a}a)^{1/2}$. 
If $a$, $b\in\Img\mathbb{H}$, then $ab=-\langle a,b\rangle+a\times b$, where 
$\times$ is the cross product. 
Let $S^2$ be the sphere of radius one centered at the origin in $\Img\mathbb{H}$. 
Then, $S^2=\{a\in\Img\mathbb{H}\,|\,a^2=-1\}$. Hence, $S^2$ is the set of all square roots of $-1$ in $\Img\mathbb{H}$. 
We denote by $S^3$ the three-sphere with radius one centered at the origin in $\mathbb{H}$. 

Fix a map $N\colon M\to S^{2}$. 
We use $N$ instead of $i$ for the decomposition of a one-form with values in $\mathbb{H}$. 
Because the multiplication in $\mathbb{H}$ is not commutative, two kinds of a one-form with values in $\mathbb{H}$ play a role of a one-form with values in $\mathbb{C}$ of type $(1,0)$ as follows. 

Let $\omega$ be a one-form with values in $\mathbb{H}$ on $M$. 
We define a one-form $\omega_N$ and a one-form $\omega^N$ by setting 
\begin{gather*}
\omega_N:=\frac{1}{2}(\omega-N\ast\,\omega),\enskip 
\omega^N:=\frac{1}{2}(\omega-\ast\,\omega\,N). 
\end{gather*}
Then, $\omega$ decomposes because
$\omega=\omega_N+\omega_{-N}=\omega^N+\omega^{-N}$. 
We see that $\ast\,\omega_N=N\,\omega_N$ and $\ast\,\omega^N=\omega\,N$. 
Clearly, $\omega=\omega_N$ if and only if $\omega_{-N}=0$. 
Similarly, $\omega=\omega^N$ if and only if $\omega^{-N}=0$. 
The quaternionic conjugation provides an identity
$\overline{\omega_N}=\overline{\omega}^{-N}$. 
We have the following decomposition of a two-form. 
\begin{lemma}\label{type}
Let $\omega$ and $\eta$ be one-forms with values in $\mathbb{H}$ on $M$. 
Then 
\begin{gather*}
\omega\wedge\eta=\omega^N\wedge\eta_{-N}+\omega^{-N}\wedge\eta_{N}.
\end{gather*}
\end{lemma}
\begin{proof}
Because $\ast\,\omega^N=\omega^N\,N$ and $\ast\,\omega=N\,\omega_{N}$, we have 
$\omega^N\wedge\eta_N=\omega^{-N}\wedge\eta_{-N}=0$ (see \cite{BFLPP02}, Proposition 16). 
Then, 
\begin{gather*}
\omega\wedge\eta=(\omega^{N}+\omega^{-N})\wedge(\eta_N+\eta_{-N})=\omega^N\wedge\eta_{-N}+\omega^{-N}\wedge\eta_{N}.
\end{gather*}
Hence, the lemma holds. 
\end{proof}
We regard $\mathbb{C}$ as a subset $\{a_0+a_1i\in\mathbb{H}\,|\,a_0,a_1\in\mathbb{R}\}$ 
of $\mathbb{H}$. 
Then, $\mathbb{H}$ is considered as a left complex vector space $\mathbb{C}\oplus \mathbb{C}j$ or 
a right complex vector space $\mathbb{C}\oplus j\mathbb{C}$.
Let $z$ be the standard holomorphic coordinate of $\mathbb{C}$ and $(x,y)$ the real coordinate 
such that $z=x+yi$. Then, $\ast\,dx=-dy$. For a map 
$f\colon \mathbb{C}\to\mathbb{H}$, we have 
\begin{gather*}
(df)_N=\frac{1}{2}\left[\frac{\partial f}{\partial x}dx+\frac{\partial f}{\partial y}dy-N\ast\left(\frac{\partial f}{\partial x}dx+\frac{\partial f}{\partial y}dy\right)\right]\\
=\frac{1}{2}\left[\frac{\partial f}{\partial x}dx-\frac{\partial f}{\partial y}\ast dx-N\left(\frac{\partial f}{\partial x}\ast dx+\frac{\partial f}{\partial y}dx\right)\right]\\
=\frac{1}{2}\left[\left(\frac{\partial f}{\partial x}-N\frac{\partial f}{\partial y}\right)dx
-N\left(\frac{\partial f}{\partial x}-N\frac{\partial f}{\partial y}\right)\ast dx\right]\\
=\frac{1}{2}(dx-N\ast dx)\left(\frac{\partial f}{\partial x}-N\frac{\partial f}{\partial y}\right)=(dx)_N\left(\frac{\partial f}{\partial x}-N\frac{\partial f}{\partial y}\right).
\end{gather*}
Similarly, 
\begin{gather*}
(df)^N=\left(\frac{\partial f}{\partial x}-\frac{\partial f}{\partial y}N\right)(dx)^N.
\end{gather*}

A function $h\colon M\to\mathbb{C}$ is holomorphic if and only if $\bar\partial h=(dh)_{-i}=(dh)^{-i}=0$. 
Let $\lambda\colon M\to\mathbb{H}$ be a map. 
If $(d\lambda)_{-i}=0$, then $\lambda=\lambda_0+\lambda_1j$ with holomorphic functions 
$\lambda_0\colon M\to\mathbb{C}$ and $\lambda_1\colon M\to\mathbb{C}$. 
If $(d\lambda)^{-i}=0$, then $\lambda=\lambda_0+j\lambda_1$ with holomorphic functions 
$\lambda_0\colon M\to\mathbb{C}$ and $\lambda_1\colon M\to\mathbb{C}$.

A non-constant map $f\colon M\to\mathbb{H}$ is called a conformal map with left normal $N\colon M\to S^2$ if 
$(df)_{-N}=0$ (\cite{PP98}, Definition 2.1). 
Taking the quaternionic conjugate, we have $(d\overline{f})^{N}=0$. 
A map $f\colon M\to\mathbb{H}$ is called a conformal map with right normal $N\colon M\to S^2$ if $(df)^N=0$ (\cite{BFLPP02}, Definition 2). 
A holomorphic function $h\colon M\to\mathbb{C}$ is a conformal map with left normal $i$ and right normal $-i$. 

Assume that $N$ is a constant map and 
define a complex structure $J$ of $\mathbb{H}$ by the 
left regular representation of $N$ on $\mathbb{H}$, that is 
$Ja:=Na$ for each $a\in\mathbb{H}$. 
Let $(df)_{-N}=0$. Then, 
\begin{gather*}
df+J\ast\,df=df+N\ast\,df=2(df)_{-N}=0.
\end{gather*} 
Hence $f$ is holomorphic with respect to a complex structure $J$. 
Similarly, if a complex structure $J$ of $\mathbb{H}$ is defined by the 
right regular representation of $-N$ on $\mathbb{H}$, that is 
$Ja:=-aN$ for each $a\in\mathbb{H}$, and $(df)^N=0$, 
then $f$ is holomorphic with respect to a complex structure $J$. 

We recall a zero of a conformal map (\cite{PP98}). 
Let $U$ be a coordinate neighborhood of $M$, $p\in U$ 
and $z$ a holomorphic coordinate on $U$ centered at $p$. 
A map $f\colon U\to\mathbb{H}$ 
vanishes to order at least $n\geq 0$ at $p$ if 
$|f(z)|\leq C|z|^{n}$ for some constant $C>0$. 
If $f$ 
vanishes to order at least $n\geq 0$ at $p$, 
but $f$ does not vanish to order at least $n+1$ at $p$, 
then a map $f$ vanishes to order $n$ at $p$. 
The order $n$ depends only on $f$. 

For an alternate explanation of a zero of a conformal map, 
we induce a map $\psi\colon M\to \mathbb{H}$ as follows. 
Let $V=M\times \mathbb{H}$ with the projection $\pi\colon V\to M$ be the trivial right quaternionic line bundle over $M$. 
The left regular representation of $N$ on $\mathbb{H}$ determines 
an eigenbundle $E\subset V$ with eigenvalue $+i$ 
because $N^2=-1$. 
\begin{lemma}\label{psi} 
For a map $N\colon M\to N(M)\subsetneq S^2$, 
there exists $a\in S^3$ such that 
$\psi=Na+ai\colon M\to\mathbb{H}$ is a global trivializing section of $E$. 
\end{lemma}
\begin{proof}
For $a\in S^3$, 
the map $c\mapsto aca^{-1}$ is a Euclidean motion in $\Img\mathbb{H}$. 
Hence, there exists $a\in S^3$ such that 
$-aia^{-1}\neq N(p)$ for any $p\in M$. 
Then $\psi=Na+ai$ does not vanish and 
$N\psi=\psi i$. 
\end{proof}

We assume that $f\colon U\to\mathbb{H}$ is a conformal map with 
left normal $N$ such that $N(U)\subsetneq S^2$. 
By Lemma \ref{psi}, Theorem 3.2 in \cite{PP98} and Lemma 3.9 in \cite{FLPP01}, 
there exist a global trivializing section $\psi$ of $E$, 
a nowhere-vanishing map $\phi_f\colon U\to\mathbb{H}$, and 
a map $\xi\colon U\to\mathbb{H}$
such that 
\begin{gather*}
f(z)=\psi(z)(z^{n}\phi_f(z)+\xi(z)),\enskip \overline{\lim_{z\to 0}}\frac{|\xi(z)|}{|z|^{n+1}} <\infty.
\end{gather*}
The point $p$ is called a zero of $f$. 
The integer $n$ is called the order of $f$ at $p$ and denoted by $\ord_p f$. 
We see that a zero of a conformal map is an isolated point. 

As an analog of a zero of a conformal map, we introduce the notion of a pole of a conformal map. 
\begin{lemma}\label{lem:pole}
Let $U$ be a coordinate neighborhood of $M$, $p\in U$ and $z$ be a holomorphic coordinate on $U$ centered at $p$. 
Let $f\colon U\setminus\{p\}\to\mathbb{H}$ be a conformal map with left normal $N$ such that 
$N(U)\subsetneq S^2$ and 
$\psi$ be a global trivializing section of $E$.  
We assume that $|f(z)|\leq C|z|^{-n}$ for some constant $C>0$, 
but there does not exist a constant $\tilde{C}>0$ such that 
$|f(z)|\leq \tilde{C}|z|^{-n+1}$ for a positive integer $n$. 

Then there exists a nowhere-vanishing map $\phi_f\colon U\to\mathbb{H}$ and a map $\xi\colon U\to\mathbb{H}$ such that 
\begin{gather}
f(z)=\psi(z)\left(z^{-n}\phi_f(z)+\xi(z)\right),\enskip \overline{\lim_{z\to 0}}\frac{|\xi(z)|}{|z|^{-n+1}} <\infty.\label{pole}
\end{gather}
\end{lemma}
\begin{proof}
We give a proof which is parallel to the proof of Lemma 3.9 in \cite{FLPP01}. 
Because $|f(z)|\leq C|z|^{-n}$, the map 
$z^{n}(\psi(z))^{-1}f(z)$ is defined on $U$. 
Let us define 
$\lambda:=\lambda_0+\lambda_1j$ with 
$\lambda_0$, $\lambda_1\colon U\to\mathbb{C}$ by 
$f(z)=\psi(z) z^{-n}\lambda(z) $. 
Then, the equation $(df)_{-N}=0$ becomes 
\begin{gather}
(d\psi)_{-N}z^{-n}\lambda(z) +\psi(z) z^{-n}(d\lambda)_{-i}=0. \label{pzl}
\end{gather}
Let $\alpha_0$ and $\alpha_1$ be complex one-forms on $U$ such that 
$(d\psi)_{-N}=\psi(\alpha_0+\alpha_1j)$. 
Because 
\begin{gather*}
\ast (d\psi)_{-N}=-N(d\psi)_{-N}=-N\psi(\alpha_0+\alpha_1j)\\
=-\psi i(\alpha_0+\alpha_1j)=\psi((-i)\alpha_0+(-i)\alpha_1j)
\end{gather*}
and 
\begin{gather*}
\ast (d\psi)_{-N}=\psi(\ast\,\alpha_0+\ast\,\alpha_1j), 
\end{gather*}
the one-forms $\alpha_0$ and $\alpha_1$ are of type $(0,1)$ with respect to $i$. 
The equation \eqref{pzl} becomes
\begin{gather*}
\alpha_0\,z^{-n}\lambda_0(z) -\alpha_1\,\overline{z}^{-n}\overline{\lambda_1(z)}
+ z^{-n}\overline{\partial}\lambda_0=0,\\
\alpha_0\,z^{-n}\lambda_1(z) +\alpha_1\,\overline{z}^{-n}\overline{\lambda_0(z)} +z^{-n}\overline{\partial}\lambda_1=0.
\end{gather*}
Simplifying this system of equations, we have 
\begin{gather*}
\alpha_0\,\lambda_0(z) -\alpha_1\,\left(\frac{z}{\overline{z}}\right)^n\overline{\lambda_1(z)} +\overline{\partial}\lambda_0 =0,\\
\alpha_0\,\lambda_1(z) +\alpha_1\,\left(\frac{z}{\overline{z}}\right)^n\overline{\lambda_0(z)}+\overline{\partial}\lambda_1=0.
\end{gather*}
Hence
\begin{gather*}
\begin{pmatrix}
\bar\partial\lambda_0\\
\bar\partial\lambda_1
\end{pmatrix}
+
\begin{pmatrix}
\alpha_0&0\\
0&\alpha_0
\end{pmatrix}
\begin{pmatrix}
\lambda_0(z)\\
\lambda_1(z)
\end{pmatrix}
+
\begin{pmatrix}
0&-\alpha_1\left(z/\overline{z}\right)^n\\
\alpha_1\left(z/\overline{z}\right)^n&0
\end{pmatrix}
\begin{pmatrix}
\overline{\lambda_0(z)}\\
\overline{\lambda_1(z)}
\end{pmatrix}
=
\begin{pmatrix}
0\\
0
\end{pmatrix}
.
\end{gather*}
This system of equations has the same form as equation (51) in \cite{FLPP01}. 
Hence, there exists a nowhere-vanishing map $\phi_f\colon U\to\mathbb{H}$ and a map
$\xi\colon U\to\mathbb{H}$ such that \eqref{pole} holds for a positive integer $n$. 
\end{proof}
\begin{definition}\label{cmp}
We call a point $p$ in Lemma \ref{lem:pole} a pole of $f$ and 
the integer $n$ the order of $f$ at a pole $p$. 
For a Riemann surface $M$ which is biholomorphic to a Riemann surface $\tilde{M}$ with discrete set $\mathcal{P}$ removed, we call $f\colon M\to\mathbb{H}$ a conformal map with poles at $\mathcal{P}$ if 
each point in $\mathcal{P}$ is a pole of $f$. 
\end{definition}
A pole of a conformal map is isolated. 
If $f$ is a meromorphic function on $U$, then the left normal $N=i$ is constant. 
Choosing $a=-i/2$, we have $\psi=1$. 
Then, the order of $f$ as a conformal map is equal to that as a meromorphic function.

Recall that a divisor on $M$ is a map $D\colon M\to\mathbb{Z}$ such that, 
for any compact subset $K$ of $M$, 
the set $\{p\in M\,|\,D(p)\neq 0\}\cap K$ is a finite set (see \cite{Forster91}). 
The set $\{p\in M\,|\,D(p)\neq 0\}$ is called the support of $D$ and denoted by $\supp D$. 
The degree of a divisor $D$ is defined by $\deg D:=\sum_{p\in M}D(p)$. 
We denote by $\Div(M)$ the set of all divisors on $M$. 

We introduce the notion of the divisor of a conformal map as follows. 
Let $\mathcal{P}$ be a subset of $M$ such that, 
for any compact subset $K$ of $M$, 
the set $\mathcal{P}\cap K$ is a finite set. 
Let $f\colon M\setminus \mathcal{P}\to\mathbb{H}$ be a conformal map with left normal $N$ and poles at $\mathcal{P}$. 
We define $\ord_pf$ by 
\begin{gather*}
\ord_pf=
\begin{cases}
0,& \textrm{if $f$ is neither zero nor pole at $p$,}\\
k,&\textrm{if $f$ has a zero of order $k$ at $p$,}\\
-k,&\textrm{if $f$ has a pole of order $k$ at $p$,}\\
\infty,&\textrm{if $f$ is identically zero in a neighborhood of $p$.}
\end{cases}
\end{gather*}

We define a map $(f)\colon M\to \mathbb{Z}$ by $(f)(p):=\ord_pf$ for each $p\in M$. 
Let us define a nonnegative map $Z\colon M\to\mathbb{Z}$ by 
$Z(p)=\max\{\ord_pf,0\}$ for each $p\in M$ and 
a nonnegative map $P\colon M\to\mathbb{Z}$ by 
$P(p)=\max\{-\ord_pf,0\}$ for each $p\in M$. 
Then $(f)=Z-P$. 
The map $P$ is a divisor on $M$ by the assumption.
The map $(f)$ is a divisor on $M$ if and only if $Z$ is a divisor on $M$. 
\begin{definition}
We assume that $(f)$ is a divisor on $M$. 
We call $(f)$ the divisor of $f$ and 
the map $f$ a conformal map with divisor $(f)$. 
We call the divisors $Z$ and $P$ 
the zero divisor of $f$ and the polar divisor of $f$ respectively. 
\end{definition}

There exists an important class of conformal maps with poles. 
\begin{proposition}\label{cmsf}
Let $M$ be an open Riemann surface and $f\colon M\to\mathbb{H}$ a conformal map which is a complete minimal surface of finite total curvature with respect to the induced (singular) metric. 
Then $f$ is a conformal map with poles.  
\end{proposition}
\begin{proof}
By Chern and Osserman \cite{CO67} and Moriya \cite{Moriya98}, $M$ is biholomorphic to a closed Riemann surface with a set of  a finite number of points $\mathcal{P}=\{p_1,\dots,p_r\}$ removed. 
Let $f_m\colon M\to\mathbb{R}$ $(m=0,1,2,3)$ be a map such that $f=f_0+f_1i+f_2j+f_3k$. 
At each point $p_l\in\mathcal{P}$, there exists a meromorphic function $F_{m,l}$ at $p_l$ 
such that $\real F_{m,l}=f_{m}$ and $n_l:=-\min\{\ord_{p_l}F_{m,l}\,|\,m=0,1,2,3\}>0$ $(m=0,1,2,3,\,l=1,\dots,r)$. 

Let $z$ be a local holomorphic coordinate centered at $p_l\in\mathcal{P}$. 
Then, 
$|F_{m,l}(z)|\leq C_{m,l}|z|^{-n_{l}}$ 
for some constant $C_{m,l}>0$. 
Because $|f_m(z)|\leq |F_{m,l}(z)|$, we have $|f_{m}(z)|\leq C_{m,l}|z|^{-n_{l}}$.  
Then 
\begin{gather*}
|f(z)|\leq 
\sum_{m=0}^3|f_{m}(z)|
\leq \left(\sum_{m=0}^3C_{m,l}\right)|z|^{-n_l}. 
\end{gather*}
Hence 
$|f(z)|\leq C_{l}|z|^{-n_{l}}$ 
for some constant $C_{l}>0$. 

Because $F_{m,l}$ is meromorphic and 
\begin{gather*}
f=\real F_{0,l}+\real F_{1,l}i+\real F_{2,l}j+\real F_{3,l}k
\end{gather*} 
at $p_l$, 
there does not exist a constant $\tilde{C}_{l}>0$ such that 
$|f(z)|\leq \tilde{C}_{l}|z|^{-n_{l}+1}$. 
Thus $p_{l}$ is a pole of $f$. 
Then $f$ is a conformal map with poles. 
\end{proof}
From the proof of Proposition \ref{cmsf}, we have the following corollary immediately. 
\begin{corollary}
Let $\tilde{M}$ be a closed Riemann surface, 
$f=f_0+f_1i+f_2j+f_3k\colon \tilde{M}\setminus\{p_1,\ldots,p_r\}\to\mathbb{H}$ be 
a complete minimal surface of finite total curvature 
and $F_{m.l}$ a meromorphic function at $p_l$ such that $\real F_{m,l}=f_m$ $(m=0,1,2,3,\,l=1,\dots,r)$. 
Then $\ord_{p_l}f=\min\{\ord_{p_l}F_{m,l}\,|\,m=0,1,2,3\}$ $(l=1,\dots,r)$. 
\end{corollary}

\section{Meromorphic functions}
We factor a meromorphic function on a Riemann surface. 

The map  $\st\colon S^2\setminus\{k\}\to\mathbb{C}$ 
defined by 
\begin{gather*}
\st(x_1i+x_2j+x_3k)=\frac{x_1}{1-x_3}+\frac{x_2}{1-x_3}i\enskip(x_1,x_2,x_3\in\mathbb{R})
\end{gather*}
is the stereographic projection from $k$.
We model $S^2$ on the complex projective line $\mathbb{C}P^1$ so that 
$\st$ is a holomorphic map. 

Let $N\colon M\to S^2$ be a conformal map. 
Then $N\colon M\to S^2\cong \mathbb{C}P^1$ is holomorphic or anti-holomorphic. 
Differentiating the equation $N^2=-1$, we have $dN\,N+N\,dN=0$. By Lemma 2 in \cite{BFLPP02}, we have 
$(dN)_N=0$ or $(dN)_{-N}=0$. 
\begin{lemma}\label{hol}
A map $N\colon M\to S^2\cong\mathbb{C}P^1$ is holomorphic if and only if $N$ satisfies 
$(dN)_{N}=(dN)^{-N}=0$. 
\end{lemma}
\begin{proof}
For a map $N\colon M\to S^2$, we define real-valued functions $n_1$, $n_2$, and $n_3$ by $N=n_1i+n_2j+n_3k$. 
Put $M_+:=M\setminus\{p\in M\,|\,N(p)=k\}$ and 
$M_-:=M\setminus\{p\in M\,|\,N(p)=-k\}$. 

We assume that $(dN)_{N}=(dN)^{-N}=0$. 
This equation becomes 
\begin{gather*}
n_1\,\ast\,dn_1+n_2\,\ast\,dn_2-n_3\,\ast\,dn_3=0,\\
dn_1-n_2\,\ast\,dn_3+n_3\,\ast\,dn_2=0,\\
dn_2-n_3\,\ast\,dn_1+n_1\,\ast\,dn_3=0,\\
dn_3-n_1\,\ast\,dn_2+n_2\,\ast\,dn_1=0.
\end{gather*}
We define functions $l_1$ and $l_2$ with values in $\mathbb{R}$ on $M_+$ by 
\begin{gather*}
l_1+l_2i:=\st\circ N=\frac{n_1}{1-n_3}+\frac{n_2}{1-n_3}i. 
\end{gather*}
Then, 
\begin{gather*}
dl_1=\frac{dn_1(1-n_3)-n_1\,d(1-n_3)}{(1-n_3)^2}=\frac{dn_1-dn_1\,n_3+n_1\,dn_3}{(1-n_3)^2}
=\frac{dn_1+\ast\,dn_2}{(1-n_3)^2},\\
dl_2=\frac{dn_2-dn_2\,n_3+n_2\,dn_3}{(1-n_3)^2}=\frac{dn_2-\ast\,dn_1}{(1-n_3)^2}. 
\end{gather*}
Hence, 
\begin{gather*}
(d(l_1+l_2i))_{-i}=0.
\end{gather*}
Then, $l_1+l_2i$ is a holomorphic function. 
Hence, $N|_{M_+}$ is a holomorphic map.

The map 
\begin{gather*}
-iNi=n_1i-n_2j-n_3k
\end{gather*}
is a rotation of $N$ centered at the origin. 
We have $M\setminus\{p\in M\,|\,-i(N(p))i=k\}=M_-$. 
In an analogous discussion as above, we show that 
$-iNi|_{M_-}$ is a holomorphic map. 
This is equivalent to having 
$N|_{M_-}$ as a holomorphic map. 
Therefore, $N$ is a holomorphic map on $M$. 

Conversely, we assume that $N$ is holomorphic. Then, $(d(l_1+l_2i))_{-i}=0$. 
Because
\begin{gather*}
N=\frac{2l_1}{l_1^2+l_2^2+1}i+\frac{2l_2}{l_1^2+l_2^2+1}j+\frac{l_1^2+l_2^2-1}{l_1^2+l_2^2+1}k,
\end{gather*}
we have 
\begin{gather*}
dN=\left(\frac{2(-l_1^2+l_2^2+1)}{(l_1^2+l_2^2+1)^2}dl_1-\frac{4l_1l_2}{(l_1^2+l_2^2+1)^2}dl_2\right)i\\
+\left(\frac{2(l_1^2-l_2^2+1)}{(l_1^2+l_2^2+1)^2}dl_1-\frac{4l_1l_2}{(l_1^2+l_2^2+1)^2}dl_2\right)j\\
+\left(\frac{4l_1}{(l_1^2+l_2^2+1)^2}dl_1+\frac{4l_2}{(l_1^2+l_2^2+1)^2}dl_2\right)k,\\
N\,dN=\left(\frac{4l_1l_2}{(l_1^2+l_2^2+1)^2}dl_1+\frac{2(-l_1^2+l_2^2+1)}{(l_1^2+l_2^2+1)^2}dl_2\right)i\\
+\left(\frac{4l_1l_2}{(l_1^2+l_2^2+1)^2}dl_1+\frac{2(l_1^2-l_2^2+1)}{(l_1^2+l_2^2+1)^2}dl_2\right)j\\
+\left(-\frac{4l_2}{(l_1^2+l_2^2+1)^2}dl_1+\frac{4l_1}{(l_1^2+l_2^2+1)^2}dl_2\right)k. 
\end{gather*}
Hence, $(dN)_{N}|_{M_+}=0$. 
A similar discussion for $-iNi$ shows that $(dN)_{N}|_{M_-}=0$. 
Hence, $(dN)_{N}=(dN)^{-N}=0$ over $M$. 
\end{proof}
\begin{corollary}\label{antihol}
A map $N\colon M\to S^2\cong\mathbb{C}P^1$ is anti-holomorphic if and only if $-N$ is holomorphic. 
In other words, a map $N\colon M\to S^2$ is anti-holomorphic if and only if
$(dN)_{-N}=(dN)^N=0$. 
\end{corollary}
\begin{proof}
We have 
\begin{gather*}
\st\circ (-N)=\frac{-(n_1+n_2i)}{1+n_3}
=\frac{-(n_1^2+n_2^2)}{(1+n_3)(n_1-n_2i)}\\
=\frac{-(1-n_3^2)}{(1+n_3)(n_1-n_2i)}
=\frac{-(1-n_3)}{n_1-n_2i}
=-\frac{1}{\overline{\st\circ N}}.
\end{gather*}
Hence, $N$ is anti-holomorphic if and only if $-N$ is holomorphic; that is 
$(dN)_{-N}=(dN)^N=0$. 
\end{proof}
\begin{lemma}\label{fctmerom}
Let $M$ be a Riemann surface.

$(1)$ A map $N\colon M\to \mathbb{C}P^1\cong S^2$ is a holomorphic map if and only if, for 
any point $p\in M$, 
there exist holomorphic functions $\lambda_0$ and $\lambda_1$ at $p$ such that $\lambda_0$ and $\lambda_1$ have no common zero and that $N=(\lambda_0+j\lambda_1) i(\lambda_0+j\lambda_1)^{-1}$. 

$(2)$ A map $N\colon M\to \mathbb{C}P^1\cong S^2$ is an anti-holomorphic map if and only if, for any point $p\in M$, 
there exist holomorphic functions $\lambda_0$ and $\lambda_1$ at $p$ such that $\lambda_0$ and $\lambda_1$ have no common zero and that $N=-(\lambda_0+j\lambda_1) i(\lambda_0+j\lambda_1)^{-1}$.  \end{lemma}
\begin{proof}
$(2)$ follows from $(1)$ and Corollary \ref{antihol}. 
We show $(1)$. 

We assume that $\lambda_0$ and $\lambda_1$ are holomorphic functions at $p$ without common zero. 
Set $\lambda:=\lambda_0+j\lambda_1$. 
Then $(d\lambda)^{-i}=0$. 
We have  
\begin{gather*}
dN=d\lambda\, i \lambda^{-1}-\lambda i\lambda^{-1}\,d\lambda\,\lambda^{-1},\\
N\ast\,dN=\lambda i\lambda^{-1} \ast\,d\lambda\, i \lambda^{-1}+\ast\,d\lambda\,\lambda^{-1}
=-\lambda i\lambda^{-1} d\lambda\, \lambda^{-1}+d\lambda\,i\lambda^{-1}. 
\end{gather*}
Hence, $(dN)_N=0$. Then, $N$ is holomorphic at $p$ by Lemma \ref{hol}. 

Conversely, we assume that $N$ is holomorphic at $p$. Then, $(dN)_N=0$. 
For any $a\in S^3$, the quaternion $aia^{-1}$ is a rotation of $i$ centered at the origin in $\Img\mathbb{H}$. 
Hence, there exists a map $\xi$ with $|\xi|=1$ such that $N=\xi i \xi^{-1}$. 

The equation $(dN)_N=0$ becomes
\begin{gather*}
d\xi\, i \xi^{-1}-\xi i\xi^{-1}\,d\xi\,\xi^{-1}
=\xi i\xi^{-1} \ast d\xi\, i \xi^{-1}+\ast\,d\xi\,\xi^{-1}. 
\end{gather*}
Simplifying this equation, we have 
\begin{gather*}
i\xi^{-1}(d\xi)^{-i}
=\xi^{-1}(d\xi)^{-i}\,i. 
\end{gather*} 
Hence $\omega:=\xi^{-1}(d\xi)^{-i}$ is a complex $(0,1)$-form. 
Let $\xi_0$ and $\xi_1$ be complex functions such that $\xi=\xi_0+j\xi_1$. 
Because $|\xi|=1$, the functions $\xi_0$ and $\xi_1$ have no common zero. 
Then, 
\begin{gather*}
\bar\partial\xi_0+j\,\bar\partial\xi_1=\xi_0\omega+j\xi_1\omega. 
\end{gather*}
Hence, 
\begin{gather*}
\omega=\bar\partial\log\xi_0=\bar\partial\log\xi_1. 
\end{gather*}
Then, there exist holomorphic functions $\lambda_0$ and $\lambda_1$ at $p$ without common zeros at $p$ such that $\lambda_0\xi_0=\lambda_1\xi_1$. 
Then 
\begin{gather*}
N=\xi i\xi^{-1}=(\xi_0+j\xi_1)i(\xi_0+j\xi_1)^{-1}\\
=(\lambda_0\xi_0+j\lambda_0\xi_1)\lambda_0^{-1}i\lambda_0(\lambda_0\xi_0+j\lambda_0\xi_1)^{-1}\\
=(\lambda_1\xi_1+j\lambda_0\xi_1)i(\lambda_1\xi_1+j\lambda_0\xi_1)^{-1}\\
=(\lambda_1+j\lambda_0)i(\lambda_1+j\lambda_0)^{-1}. 
\end{gather*}
Thus, the theorem holds. 
\end{proof}
If $N=(\lambda_0+j\lambda_1) i(\lambda_0+j\lambda_1)^{-1}$ is a holomorphic map with 
holomorphic functions $\lambda_0$ and $\lambda_1$, then 
$N=\Lambda i\overline{\Lambda}$ with $\Lambda:=(\lambda_0+j\lambda_1)/|\lambda|$.  

If $M$ is a closed Riemann surface and $\lambda_0$ and $\lambda_1$ are meromorphic on $M$, 
then $N\colon M\to S^2\cong\mathbb{C}P^1$ is a holomorphic map. 
From this factorization, we have a relation between the degree of $N$ and the degree of $\lambda$ 
when $M$ is closed. 
\begin{theorem}\label{degree}
Let $M$ be a closed Riemann surface, $\lambda_0$ and $\lambda_1$ are meromorphic functions on $M$, 
$\lambda:=\lambda_0+j\lambda_1$, and $N:=\lambda i \lambda^{-1}$. 
The degree of a holomorphic map $N$ is $m$ if and only if the degree of $\lambda_0$ and $\lambda_1$ are $m$. 
\end{theorem}
\begin{proof}
We assume that the degree of $N$ is $m$. 
The equation $N=i$ has $m$ solutions counting multiplicities. 
This equation is equivalent to the equation $i\lambda=\lambda i$. 
Rewriting this equation, we have $\lambda_0 i=\lambda_0 i$ and $-\lambda_1 i=\lambda_1i$. 
The former equation is trivial; the latter is equivalent to $\lambda_1=0$. 
Hence the equation $\lambda_1=0$ has $m$ solutions counting multiplicities.
Then, $\lambda_1$ is a meromorphic function of degree $m$. 
Next, we consider the equation $N=-i$. This equation is equivalent to 
$\lambda_0=0$. 
Hence $\lambda_0$ is a meromorphic function of degree $m$. 
The converse is trivial. 
\end{proof}
\section{Super-conformal maps}
We factor a super-conformal map.

We recall the definition and basic properties of a super-conformal map (see \cite{BFLPP02}). 
A conformal map $f\colon M\to\mathbb{H}$ is called a super-conformal map if its curvature ellipse is a circle. 
A conformal map $f$ is super-conformal if and only if 
its left normal or its right normal is anti-holomorphic. 
Let $N\colon M\to S^2$ be the left normal of $f$ and $R\colon M\to S^2$ the right normal of $f$. 
Then $f$ is super-conformal if and only if $(dN)_{-N}=0$ or 
$(dR)_{-R}=0$ by Corollary \ref{antihol}. 
Then the following is trivial by Lemma \ref{hol} and Corollary \ref{antihol}:
\begin{lemma}\label{mersc}
A holomorphic map and an anti-holomorphic map from $M$ to $\mathbb{C}P^1\cong S^2$ are super-conformal.
\end{lemma}

It is known that a super-conformal map is a stereographic projection composed with  
the twistor projection of a holomorphic map from a Riemann surface to $\mathbb{C}P^3$ (\cite{BFLPP02}, Theorem 5). 
Hence, for holomorphic functions $\lambda_0$, $\lambda_1$, $\lambda_2$ and $\lambda_3$, 
a map 
\begin{gather}
f=(\lambda_0+\lambda_1j)^{-1}(\lambda_2+\lambda_3j)\label{sctw}
\end{gather}
is a super-conformal map with anti-holomorphic left normal. 
Indeed, 
\begin{gather*}
df=-(\lambda_0+\lambda_1j)^{-1}(d\lambda_0+d\lambda_1j)(\lambda_0+\lambda_1j)^{-1}(\lambda_2+\lambda_3j)\\
+(\lambda_0+\lambda_1j)^{-1}(d\lambda_2+d\lambda_3j). 
\end{gather*}
Then 
\begin{gather*}
\ast\,df=(\lambda_0+\lambda_1j)^{-1}i(\lambda_0+\lambda_1j)\,df. 
\end{gather*}
Hence $f$ is conformal with left normal $(\lambda_0+\lambda_1j)^{-1}i(\lambda_0+\lambda_1j)$. 
We have 
\begin{gather*}
(\lambda_0+\lambda_1j)^{-1}i(\lambda_0+\lambda_1j)
=(\overline{\lambda_0}-\lambda_1j)i(\overline{\lambda_0}-\lambda_1j)^{-1}\\
=(-\overline{\lambda_0}j-\lambda_1)ji(-j)(-\overline{\lambda_0}j-\lambda_1)^{-1}\\
=-(\lambda_1+j\lambda_0)i(\lambda_1+j\lambda_0)^{-1}. 
\end{gather*}
By Lemma \ref{fctmerom}, the map $-(\lambda_1+j\lambda_0)i(\lambda_1+j\lambda_0)^{-1}$ is 
an anti-holomorphic map. 
Hence $f$ is a super-conformal map with anti-holomorphic left normal. 
We can consider \eqref{sctw} as a factorization of a super-conformal map. 
Conversely, 
if an anti-holomorphic map $-(\lambda_1+j\lambda_0)i(\lambda_1+j\lambda_0)^{-1}$ is given, then 
a map $f=(\lambda_0+\lambda_1j)^{-1}(\lambda_2+\lambda_3j)$ with holomorphic functions 
$\lambda_2$ and $\lambda_3$ is a super-conformal map.

We have another factorization of a super-conformal map. 
Let $V=M\times \mathbb{H}$ with the projection 
$\pi\colon V\to M$ be the trivial right quaternionic 
line bundle over $M$ 
and $N\colon M\to N(M)\subsetneq S^2\cong\mathbb{C}P^1$ an 
anti-holomorphic map. 
Let $E\subset V$ be the eigenbundle of 
the left regular representation of $N$ on $\mathbb{H}$ with eigenvalue $+i$.  
Lemma \ref{psi} ensures 
that there exists a global trivializing section $\psi$ of $E$. 
\begin{lemma}\label{gscs}
Let $V=M\times \mathbb{H}$ with the projection 
$\pi\colon V\to M$ be the trivial right quaternionic 
line bundle over $M$, 
$N\colon M\to N(M)\subsetneq S^2\cong\mathbb{C}P^1$ be an 
anti-holomorphic map and 
$E\subset V$ be the eigenbundle  of the left regular representation of $N$ on $\mathbb{H}$ with eigenvalue $+i$.
Then there exists a global trivializing section $\psi\colon M\to\mathbb{H}$ of $E$ which is a super-conformal map with anti-holomorphic left normal $N$. 
\end{lemma}
\begin{proof}
From the proof of Lemma \ref{psi}, there exists $a\in S^3$ such that 
$\psi_0:=Na+ai$ is a global trivializing section of $E$. 
Because $d\psi_0=dN\,a$, the map $\psi_0$ is a super-conformal map 
with left normal $N$. 
\end{proof}
\begin{definition}
We call a global trivializing section $\psi$ of $E$ which is a super-conformal map with anti-holomorphic left normal $N$ a global super-conformal trivializing section of $E$. 
\end{definition}
\begin{theorem}[Factorization theorem for super-conformal maps]\label{fctsc}
Let $V=M\times \mathbb{H}$ with the projection 
$\pi\colon V\to M$ be the trivial right quaternionic 
line bundle over $M$, 
$N\colon M\to N(M)\subsetneq S^2\cong\mathbb{C}P^1$ be an 
anti-holomorphic map, 
$E\subset V$ be the eigenbundle of the left regular representation of $N$ on $\mathbb{H}$ with eigenvalue $+i$ and 
$\psi$ be a global super-conformal trivializing section of $E$. 
A map $f\colon M\to\mathbb{H}$ is a super-conformal map with anti-holomorphic left normal $N$ 
if and only if 
$f=\psi(\lambda_0+\lambda_1j)$ with holomorphic functions $\lambda_0$ and $\lambda_1$ on $M$. 
\end{theorem}

\begin{proof}
Because $\psi$ is nowhere-vanishing, any map $f\colon M\to\mathbb{H}$ is 
factored by the product $\psi(\lambda_0+\lambda_1j)$ with complex functions $\lambda_0$ and $\lambda_1$ on $M$. 
Let $\lambda:=\lambda_0+\lambda_1j$. The functions $\lambda_0$ and $\lambda_1$ are holomorphic if and only if $(d\lambda)_{-i}=0$. 
We have 
\begin{gather*}
2(d(\psi\lambda))_{-N}
=d\psi\,\lambda+\psi\,d\lambda+N\ast(d\psi\,\lambda+\psi\,d\lambda)\\
=(d\psi+N\ast\,d\psi)\lambda+\psi\,d\lambda+N\psi\ast d\lambda
=\psi\,d\lambda+\psi i\ast d\lambda
=2\psi(d\lambda)_{-i}. 
\end{gather*}
Hence, $f=\psi\lambda$ with 
$(d\lambda)_{-i}=0$ if and only if 
$f$ is super-conformal with anti-holomorphic left normal $N$. 
\end{proof}
By the above theorem, 
the set of super-conformal maps from $M$ to $\mathbb{H}$ with anti-\hspace{0pt}holomorphic left normal $N\colon M\to N(M)\subsetneq S^2\cong\mathbb{C}P^1$ is parametrized by two holomorphic functions.
Hence, a super-conformal map adopts properties of a holomorphic function. 
 
The following is an analog of the Liouville theorem  (see \cite{GK06}). 
\begin{theorem}[Liouville's theorem for super-conformal maps]\label{Liouville}
Let $f\colon \mathbb{C}\to\mathbb{H}$ be a super-conformal map 
with anti-holomorphic left normal $N\colon \mathbb{C}\to N(\mathbb{C})\subsetneq S^2\cong\mathbb{C}P^1$ and 
$\psi$ be a global super-conformal trivializing section of $E$. 
If $|\psi|^{-1}$ 
and $|f|$ are bounded above, then 
$f=\psi C$ with a constant $C\in\mathbb{H}$. 
\end{theorem}
\begin{proof}
Because $|\psi|^{-1}$ is bounded above, 
there exists a constant $c>0$ such that 
$|\psi|^{-1}\leq c$. 
By the factorization theorem for super-conformal maps, 
$f=\psi(\lambda_0+\lambda_1j)$ 
with holomorphic functions $\lambda_0$ and $\lambda_1$ on $\mathbb{C}$. 
Then 
\begin{gather*}
|\lambda_0+\lambda_1j|=\left(|\lambda_0|^2+|\lambda_1|^2\right)^{1/2}=\frac{|f|}{|\psi|}\leq c|f|. 
\end{gather*}
Because $|f|$ is bounded above, holomorphic functions $\lambda_0$ and $\lambda_1$ are bounded entire functions. 
By the Liouville theorem, $\lambda_0$ and $\lambda_1$ are constant. 
\end{proof}

By Liouville's theorem for super-conformal maps, we see that if $|f|$ and $|\psi|^{-1}$ are bounded above, then the case where the domain of $f$ is proper subset of $\mathbb{C}$ is interesting. 

Let $z$ be the standard holomorphic coordinate of $\mathbb{C}$, $(x,y)$ the real coordinate such that $z=x+yi$ and $B^2:=\{z\in\mathbb{C}\,|\,|z|<1\}$. 
The following is an analog of the Schwarz lemma (see \cite{GK06}). 
\begin{theorem}[The Schwarz lemma for super-conformal maps]\label{Schwarz}
Let $f\colon B^2\to\mathbb{H}$ be a super-conformal map 
with anti-holomorphic left normal $N\colon B^2\to N(B^2)\subsetneq S^2\cong\mathbb{C}P^1$ and 
$\psi$ be a global super-conformal trivializing section of $E$. 
We assume that $f(0)=0$ and $|f|$ is bounded above. 
Moreover, we assume that $|\psi|\leq c$ and $|\psi|^{-1}\leq \tilde{c}$. 
Let $f=\psi(\lambda_0+\lambda_1j)$ with holomorphic functions $\lambda_0$ and $\lambda_1$ on $B^2$. 
Then, there exist constants $C_0$, $C_1>0$ such that 
\begin{gather*}
|f(z)|\leq c(C_0^2+C_1^2)^{1/2}|z|. 
\end{gather*}
The equality holds if and only if the following two conditions hold:
\begin{enumerate}
\item 
$|\psi|=c$. 
\item 
There exists $z_0\in B^2\setminus\{0\}$ such that $|\lambda_n(z_0)|=C_n|z_0|$ $(n=0,1)$. 
\end{enumerate}
Moreover, we have 
\begin{gather*}
\left|\frac{\partial f}{\partial x}(0)-N(0) \frac{\partial f}{\partial y}(0)\right|\leq c(C_0^2+C_1^2)^{1/2}.
\end{gather*}
The equality holds if and only if the following two conditions hold:
\begin{enumerate}
\item 
$|\psi(0)|=c$. 
\item 
There exists $z_0\in B^2\setminus\{0\}$ such that $|\lambda_n(z_0)|=C_n|z_0|$ $(n=0,1)$. 
\end{enumerate}
\end{theorem}
\begin{proof}
By the factorization theorem for super-conformal maps, 
$\lambda_0$ and $\lambda_1$ are holomorphic functions 
on $B^2$. 
Because $f(0)=0$ and $\psi$ is nowhere-vanishing, we have $\lambda_0(0)=\lambda_1(0)=0$. 
Also, because $|\psi|^{-1}$ 
and  
$|f|$ are bounded above, $|\psi^{-1}f|$ is bounded above. 
Because $|\psi^{-1}f|=|\lambda_0+\lambda_1j|=(|\lambda_0|^2+|\lambda_1|^2)^{1/2}$, 
the functions $|\lambda_0|$ and $|\lambda_1|$ are bounded above. 
Let $|\lambda_n|\leq C_n$ $(n=0,1)$. 
By the Schwarz lemma (for holomorphic maps), we have $|\lambda_n(0)|\leq C_n|z|$ and $|(\partial \lambda_n/\partial z)(0)|\leq C_n$. The existence of a point $z_0\in B^2\setminus\{0\}$ such that 
$|\lambda_n(z_0)|= C_n|z_0|$ is equivalent to   
$\lambda_n(z)=C_ne^{\theta_n i}z$ with real-valued function $\theta_n$. 
Then we have 
\begin{gather*}
|f(z)|=|\psi(z)|(|\lambda_0(z)|^2+|\lambda_1(z)|^2)^{1/2}\leq c(C_0^2+C_1^2)^{1/2}|z|.
\end{gather*}
The equality holds if and only if (1) $|\psi|=c$ and (2) 
there exists $z_0\in B^2\setminus\{0\}$ such that $|\lambda_n(z_0)|=C_n|z_0|$ $(n=0,1)$.  

Let $\lambda:=\lambda_0+\lambda_1j$. 
For derivatives of $f$, we have 
\begin{gather*}
\left|\frac{\partial f}{\partial x}(0)-N(0)\frac{\partial f}{\partial y}(0)\right|\\
=\left|\left(\frac{\partial \psi}{\partial x}(0)-N(0)\frac{\partial \psi}{\partial y}(0)\right)\lambda(0)
+\psi(0)\left(\frac{\partial \lambda}{\partial x}(0)-i\frac{\partial \lambda}{\partial y}(0)\right)\right|\\
=\left|\psi(0)\left(\frac{\partial \lambda}{\partial x}(0)-i\frac{\partial \lambda}{\partial y}(0)\right)\right|
=|\psi(0)|\left|\frac{\partial \lambda_0}{\partial z}(0)+\frac{\partial \lambda_1}{\partial z}(0)j\right|\\
\leq |\psi(0)|\left(\left|\frac{\partial \lambda_0}{\partial z}(0)\right|^2+\left|\frac{\partial \lambda_1}{\partial z}(0)\right|^2\right)^{1/2}\leq c(C_0^2+C_1^2)^{1/2}.
\end{gather*}
The equality holds if and only if (1) $|\psi(0)|=c$ and (2) 
there exists $z_0\in B^2\setminus\{0\}$ such that $|\lambda_n(z_0)|=C_n|z_0|$ $(n=0,1)$.  
\end{proof}
The following is an analog of the Schwarz-Pick theorem (see \cite{GK06}). 
Let $B^4:=\{a\in\mathbb{H}\,|\,|a|<1\}\subset\mathbb{H}$. 
We recall quaternionic M\"{o}bius transformations from $B^4$ to itself. 
For $a_1\in B^4$, 
the map $\Theta^{a_1}\colon B^4\to B^4$ defined by 
\begin{gather*}
\Theta^{a_1}(a)=(a-a_1)\left(1-\overline{a_1}a\right)^{-1}
\end{gather*}
is a quaternionic M\"{o}bius transformation which maps $a_1$ to $0$ by the following lemma: 
\begin{lemma}[\cite{Ahlfors81}, Section 2.6]
The quaternionic M\"{o}bius transformation $\Phi(a)=(pa+q)(ra+s)^{-1}$ with 
$p$, $q$, $r$, $s\in\mathbb{H}$ 
maps $B^4$ to itself and $\Phi(a_1)=0$ for $_1\in B^4$ 
if and only if $\Phi(a)=t(a-a_1)(1-\overline{a_1}a)^{-1}u^{-1}$ where $t$, $u\in\mathbb{H}$ with $|t|=|u|=1$. 
\end{lemma}

For $z_1\in B^2$, 
define a M\"{o}bius transform $\tau^{z_1}\colon B^2\to B^2$ by 
\begin{gather*}
\tau^{z_1}(z)=\frac{z-z_1}{1-\overline{z_1}z}. 
\end{gather*}
Let $f\colon B^2\to B^4\subset\mathbb{H}$ be a super-conformal map with anti-holomorphic left normal $N\colon B^2\to N(B^2)\subsetneq S^2$. 
For a given $z_1\in B^2$, we define a conformal map $g^{z_1}\colon B^2\to B^4$ by 
$g^{z_1}=\Theta^{f(z_1)}\circ f\circ(\tau^{z_1})^{-1}\colon B^2\to B^4$. 
Let $N^{z_1}$ be the left normal of $g^{z_1}$. 
Denote by $E^{z_1}$ the eigenbundle of the left regular representation of $N^{z_1}$ with eigenvalue $+i$. 
\begin{theorem}[The Schwarz-Pick theorem for super-conformal maps]\label{SP}
Let $f\colon B^2\to B^4\subset\mathbb{H}$ be a super-conformal map with anti-holomorphic left normal   
$N\colon B^2\to S^2$. 
Fix $z_1\in B^2$. 
Assume that $N^{z_1}(B^2)\subsetneq S^2$. 
Assume that there exists a global super-conformal trivializing section  $\psi^{z_1}$ of $E^{z_1}$ such that $|\psi^{z_1}|$ and $|\psi^{z_1}|^{-1}$ are bounded above. 
Then, there exists a constant $C^{z_1}>0$ such that  
\begin{gather*}
\frac{|f(z)-f(z_1)|}{\left|1-\overline{f(z_1)}f(z)\right|}
\leq C^{z_1}
\left|\frac{z-z_1}{1-\overline{z_1}z}\right| 
\end{gather*}
for all $z\in B^2$. 
We have 
\begin{gather*}
\frac{\left|\frac{\partial f}{\partial x}(z_1)\right|}{1-|f(z_1)|^2}
=\frac{\left|\frac{\partial f}{\partial y}(z_1)\right|}{1-|f(z_1)|^2}
\leq \frac{C^{z_1}}{1-|z_1|^2} .
\end{gather*}
\end{theorem}

\begin{proof}
Rouxel \cite{Rouxel89} showed that a conformal transform of a super-conformal map is 
a super-conformal map. 
Hence $g^{z_1}$ is a super-conformal map with $g^{z_1}(0)=0$. 
By the Schwarz lemma for super-conformal maps, there exists $C^{z_1}>0$ such that $|g^{z_1}(z)|\leq C^{z_1}\left|z\right|$. 
Hence  
\begin{gather*}
\frac{|f(z)-f(z_1)|}{\left|1-\overline{f(z_1)}f(z)\right|}
\leq C^{z_1}\left|\frac{z-z_1}{1-\overline{z_1}z}\right|. 
\end{gather*}

Let $z_1=x_1+y_1i$ and $z_2=x_2+y_1i$, $(x_1,x_2,y_1\in\mathbb{R})$. 
Then 
\begin{gather*}
\frac{\left|f(x_2+y_1i)-f(x_1+y_1i)\right|}{\left|1-\overline{f(x_1+y_1i)}f(x_2+y_1i)\right|}
\leq C^{z_1}\left|\frac{x_2-x_1}{1-\overline{(x_1+y_1i)}(x_2+y_1i)}\right|. 
\end{gather*}
Hence 
\begin{gather*}
\frac{\left|f(x_2+y_1i)-f(x_1+y_1i)\right|}{|x_2-x_1|\left|1-\overline{f(x_1+y_1i)}f(x_2+y_1i)\right|}
\leq C^{z_1}\left|\frac{1}{1-\overline{(x_1+y_1i)}(x_2+y_1i)}\right|. 
\end{gather*}
Let $x_2$ tend to $x_1$. Then 
\begin{gather*}
\frac{\left|\frac{\partial f}{\partial x}(z_1)\right|}{1-|f(z_1)|^2}\leq \frac{C^{z_1}}{1-|z_1|^2}. 
\end{gather*}
Because $f$ is conformal, we have 
\begin{gather*}
\left|\frac{\partial f}{\partial x}(z_1)\right|=\left|\frac{\partial f}{\partial y}(z_1)\right|.  
\end{gather*}
Then the theorem holds.
\end{proof}

Let $ds^2_{B^2}$ be the Poincar\'{e} metric  on $B^2$ with curvature $-1$ and 
$ds^2_{B^4}$ be the Poincar\'{e} metric on $B^4$ with curvature $-1$. 
For the standard coordinate $(x,y)$ of $\mathbb{R}^2$ and 
the standard coordinate $(a_0,a_1,a_2,a_3)$ of $\mathbb{R}^4$, we have 
\begin{gather*}
ds_{B^2}^2=\frac{4}{(1-(x^2+y^2))^2}(dx\otimes dx+dy\otimes dy),\\
ds_{B^4}^2=\frac{4}{(1-\sum_{n=0}^3a_n^2)^2}\sum_{n=0}^3(da_n\otimes da_n). 
\end{gather*}

Let $f\colon B^2\to B^4\subset\mathbb{H}$ be a super-conformal map with anti-holomorphic left normal   
$N\colon B^2\to S^2$. 
Let $P^f$ be the set of all $z\in B^2$ such that (1) $N^{z}(B^2)\subsetneq S^2$, 
(2) there exists a global super-conformal trivializing section $\psi^z$ of $E^{z}$ and (3) $|\psi^z|$ and $|\psi^z|^{-1}$ are bounded above. 
The following is a geometric interpretation of the Schwarz-Pick theorem for super-conformal maps. 
\begin{theorem}[The geometric version of the Schwarz-Pick theorem for super-conformal maps]\label{GSP}
Let $f\colon B^2\to B^4\subset\mathbb{H}$ be a super-conformal map with anti-holomorphic left normal   
$N\colon B^2\to S^2$. 
Then, at each point $z$ in $P^f$, 
there exists a constant $C^{z}>0$ such that $f^\ast ds^2_{B^4}\leq (C^{z})^2\,ds^2_{B^2}$. 
\end{theorem}
\begin{proof}
Let $f_0$, $f_1$, $f_2$ and $f_3$ be the real-valued functions such that 
$f=f_0+f_1i+f_2j+f_3k$. 
Then, 
\begin{gather*}
f^\ast ds^2_{B^4}=\frac{4}{(1-\sum_{n=0}^3(f_n(z))^2)^2}\\
\times\sum_{n=0}^3\left(\left(\frac{\partial f_n}{\partial x}(z)\right)^2dx\otimes dx+\left(\frac{\partial f_n}{\partial y}(z)\right)^2dy\otimes dy\right)\\
=\frac{4}{(1-|f(z)|^2)^2}\left(\left|\frac{\partial f}{\partial x}(z)\right|^2dx\otimes dx+\left|\frac{\partial f}{\partial y}(z)\right|^2dy\otimes dy\right). 
\end{gather*}
By the Schwarz-Pick theorem for super-conformal maps, there exists $C^{z}>0$ such that 
\begin{gather*}
\frac{4}{(1-|f(z)|^2)^2}\left|\frac{\partial f}{\partial x}(z)\right|^2=\frac{4}{(1-|f(z)|^2)^2}\left|\frac{\partial f}{\partial y}(z)\right|^2
\leq 
\frac{4(C^z)^2}{(1-|z|^2)^2}
\end{gather*}
at each $z\in P^f$. 
Hence
\begin{gather*}
f^\ast ds^2_{B^4}\leq \frac{4(C^z)^2}{(1-|z|^2)^2}(dx\otimes dx+dy\otimes dy)
=(C^z)^2ds^2_{B^2}
\end{gather*}
at each $z\in P^f$. 
\end{proof}
\begin{figure}[h]
\begin{gather*}
\xymatrix{
(B^2,f^\ast\real\langle\enskip,\enskip\rangle)\ar[r]^f&(\mathbb{H},\real\langle\enskip,\enskip\rangle)\\
P^f\ar[r]^{f}\ar@{^{(}->}[u]&B^4\ar@{^{(}->}[u]
}
\\
\textrm{($f$: super-conformal)},\\
\xymatrix{(P^f,ds^2_{B^2})\ar@{^{(}->}[r]&(B^2,ds^2_{B^2})
},\enskip 
\xymatrix{
(P^f,f^\ast ds^2_{B^4})\ar[r]^f&(B^4,ds^2_{B^4}),
}\\
f^\ast ds^2_{B^4}\leq (C^z)^2\,ds^2_{B^2}.
\end{gather*}
\caption{An analog of the Schwarz-Pick theorem.}
\label{corollary}
\end{figure}

We use 
the geometric version of the Schwarz-Pick theorem (for holomorphic maps) to investigate whether the Kobayashi pseudodistance on a complex manifold is a distance (Kobayashi \cite{Kobayashi67}). 
We define a pseudodistance on $f(B^2)$ in a similar way to define the Kobayashi pseudodistance by super-conformal maps. 
We show that the pseudodistance is a distance 
by the geometric version of the Schwarz-Pick theorem for super-conformal maps. 

Let $f\colon B^2\to B^4\subset\mathbb{H}$ be 
an injective super-conformal immersion with anti-holomorphic left normal 
$N\colon B^2\to N(B^2)\subsetneq S^2\cong\mathbb{C}P^1$. 
Assume that $f(0)=0$ and that there exists a 
global super-conformal trivializing section $\psi$ such that $|\psi|$ and 
$|\psi|^{-1}$ are bounded above. 
Then $P^f=B^2$. 

Given two points $p$, $q\in f(B^2)$, choose a sequence of points $p=p_0$, $p_1$, $\ldots$, 
$p_{s-1}$, $p_s=q$ in $f(B^2)$. 
We choose 
a sequence of points $a_1$, $\ldots$, $a_s$, $b_1$, $\ldots$, $b_s$ in  $B^2$ 
and a sequence of holomorphic maps $\phi_\alpha\colon B^2\to B^2$ with $\phi_\alpha(0)=0$ such that $(f\circ\phi_\alpha)(a_\alpha)=p_{\alpha-1}$, $(f\circ\phi_\alpha)(b_\alpha)=p_{\alpha}$ $(\alpha=1,\ldots,s)$. 
Let $\rho$ be the distance on $B^2$ defined by the Poincar\'{e} metric $ds^2_{B^2}$. 
Define $d_{f}(p,q)$ by the infimum of the sum $\sum_{\alpha=1}^s\rho(a_\alpha,b_\alpha)$ for all possible choices of sequences of points in $f(B^2)$, sequences of points in $B^2$ and sequences of holomorphic maps from $B^2$ to $B^2$ which fix the point $0\in B^2$.  
Then $d_{f}$ is a pseudodistance of $f(B^2)$. That is, 
$d_{f}(p,q)\geq 0$, $d_{f}(p,q)=d_{f}(q,p)$ and $d_{f}(p,q)+d_{f}(q,r)\geq d_{f}(p,r)$ for any $p$, $q$, $r\in f(B^2)$. 

\begin{theorem}\label{dist}
Let $f\colon B^2\to B^4\subset \mathbb{H}$ be an injective super-conformal immersion with anti-holomorphic left normal 
$N\colon B^2\to N(B^2)\subsetneq S^2\cong\mathbb{C}P^1$.
Assume that $f(0)=0$.  
Let $C^z$ be a positive constant such that 
$f^\ast ds^2_{B^4}\leq (
C^z)^2ds^2_{B^2}$ at $z\in B^2$. 
If there exists a constant $C>0$ such that 
$C^z\leq C$ for any $z\in B^2$,  
then $d_f$ is a distance on $f(B^2)$. 
\end{theorem}
\begin{proof}
By the assumption, 
we have 
$f^\ast ds^2_{B^4}\leq C^2ds^2_{B^2}$. 
By the geometric version of the Schwarz-Pick theorem (for holomorphic maps), we have 
\begin{gather*}
(f\circ\phi)^\ast ds^2_{B^4}=\phi^\ast f^\ast ds^2_{B^4}\leq C^2 \phi^\ast ds^2_{B^2}\leq C^2 ds^2_{B^2}
\end{gather*}
for every holomorphic map $\phi\colon B^2\to B^2$. 
Let $\sigma$ be the distance determined by $ds^2_{B^4}$. 
Then 
$\sigma((f\circ\phi)(a),(f\circ\phi)(b))\leq C\rho(a,b)$
for every $a$, $b\in B^2$ and 
every holomorphic map $\phi\colon B^2\to B^2$ with $\phi(0)=0$. 

Let $p_0$, $p_1$, $\ldots$, $p_k$, $a_0$, $a_1$, $\ldots$, $a_k$, $b_0$, $b_1$, $\ldots$, $b_k$, $\phi_0$, $\phi_1$, $\ldots$, $\phi_k$ be 
as in the definition of $d_{f}$. 
Then 
\begin{gather*}
\sigma(p,q)\leq\sum_{i=1}^k\sigma(p_{i-1},p_i)=\sum_{i=1}^k\sigma((f\circ \phi_i)(a_i),(f\circ \phi_i)(b_i))
\leq C\sum_{i=1}^k\rho (a_i,b_i)
\end{gather*}
Hence 
$\sigma(p,q)\leq  d_{f}(p,q)$  
for every $p$, $q\in f(M)$. 
Then $d_f(p,q)=0$ implies $\sigma(p,q)\leq 0$. Because $\sigma$ is a distance,  we have  $p=q$. 
Hence $d_f$ is a distance. 
\end{proof}

Let $N\colon M\to N(M)\subsetneq S^2\cong\mathbb{C}P^1$ 
be an anti-holomorphic map and 
$\psi$ be a global super-conformal trivializing section of $E$. 
Recalling the definition of a pole of a conformal map, 
the map $f:=\psi\lambda$ with $\lambda=\lambda_0+\lambda_1j$ for meromorphic functions $\lambda_0$ and $\lambda_1$ is a super-conformal map with poles. 
Hence, a super-conformal map with poles adopts properties of a meromorphic function. 

The Weierstrass factorization theorem (see \cite{Forster91}) states that, 
for a given divisor $D$, there exists a meromorphic function $h$ with $(h)=D$. 
Because a meromorphic function is a super-conformal map with left normal $i$, 
there exists a super-conformal map $f$ with left normal $i$ such that $(f)=D$. 
The map $\overline{f}$ is a super-conformal map with left normal $-i$ such that $(\overline{f})=D$. 
\begin{theorem}[The Weierstrass factorization theorem for super-conformal maps]\label{WFT}
For any divisor $D$ on $M$ and any anti-holomorphic map $N\colon M\to N(M)\subsetneq S^2\cong\mathbb{C}P^1$, 
there exists a super-conformal map $f\colon M\setminus \supp D\to\mathbb{H}$ with poles such that the left normal of $f$ is $N$ and $(f)=D$.  
\end{theorem}
\begin{proof}
By the Weierstrass factorization theorem, any divisor on $M$ is a divisor of a meromorphic function. 
Let $D$ be a divisor on $M$, $\lambda$ be a meromorphic function on $M$ with divisor $D$ and $\psi$ be a global super-conformal trivializing section of $E$. 
Then, $f:=\psi\lambda\colon M\setminus\supp D\to\mathbb{H}$ 
is a super-conformal map by the factorization theorem for super-conformal maps and $(f)=D$ by the definition of a divisor of a conformal map. 
\end{proof}
We assume that $M$ is a connected open subset of a closed Riemann surface $\tilde{M}$. 
We denote by $C_1(M)$ the set of all one-chains in $M$. 
We define a map 
$\delta\colon C_1(M)\to \Div(M)$ by, for $c\colon [0,1]\to M$, 
\begin{gather*}
(\delta(c))(p):=
\begin{cases}
1&(p=c(1)),\\
-1&(p=c(-1)),\\
0&(\textrm{otherwise}).
\end{cases}
\end{gather*}
The following is an analog of the Abel-Jacobi theorem (see \cite{Forster91}). 
\begin{theorem}[The Abel-Jacobi theorem for super-conformal maps]\label{Abel} 
Let $D$ be a divisor on $M$ with $\deg D=0$. 
Then, $D$ is the divisor of a super-conformal map from $M$ with poles and left normal $N\colon M\to N(M)\subsetneq S^2$, if and only if there exists $c\in C_1(\tilde{M})$ such that 
$\delta(c)=D$ and 
\begin{gather*}
\int_c\omega=0
\end{gather*}
for every holomorphic one-form $\omega$ on $\tilde{M}$.
\end{theorem}
\begin{proof}
By the Abel-Jacobi theorem, the divisor $D$ is a divisor of a meromorphic function $\lambda$ on $\tilde{M}$. 
Hence $f:=\psi\lambda\colon M\setminus\supp D\to\mathbb{H}$ 
is a super-conformal map by the factorization theorem for super-conformal maps. We see that 
$(f)=D$ by the definition of a divisor of a conformal map. 
\end{proof}

\section{Minimal surfaces}
We connect a conformal map with a classical surface. 

Let $f\colon M\to\mathbb{H}$ be a conformal map with $(df)_{-N}=(df)^R=0$. 
We induce a (singular) metric on $M$ from $\mathbb{H}$ by $f$. 
Consequently, the Gauss curvature $K$, the normal curvature $K^\perp$ and the mean curvature vector $\mathcal{H}$ of $f$ can be defined.
We have 
\begin{gather*}
df\,\overline{\mathcal{H}}=-N(dN)_N,\enskip \overline{\mathcal{H}}\,df=R(dR)_R,\\
K|df|^2=\frac{1}{2}(\langle \ast\,dR,R\,dR\rangle+\langle \ast\,dN,N\,dN\rangle),\\
K^\perp|df|^2=\frac{1}{2}(\langle \ast\,dR,R\,dR\rangle-\langle \ast\,dN,N\,dN\rangle). 
\end{gather*}
 (\cite{BFLPP02}, Proposition 8, Proposition 9). 
 A conformal map $f$ is minimal if and only if $N$ is holomorphic or, equivalently, $R$ is holomorphic.
Hence, if $f$ is super-conformal and minimal, then $N$ or $R$ is a constant map. 
Then, $f$ is a holomorphic map with respect to a complex structure of $\mathbb{H}$. 

We consider the class of surfaces with  $|K|=|K^\perp|$. 
We denote by $\sigma $ the area element of the two sphere with radius one. 
\begin{lemma}
Let $f\colon M\to\mathbb{H}$ be a conformal map with $(df)_{-N}=(df)^R=0$. 
If $|K|=|K^\perp|$, then $N^\ast\sigma=0$ or $R^\ast\sigma=0$.  
\end{lemma}
\begin{proof}
From the assumption, we have $K=\pm K^\perp$. 
Then $\langle \ast\,dN,N\,dN\rangle=0$ or $\langle \ast\,dR,R\,dR\rangle=0$.  
It is known that $\langle \ast\,dN,N\,dN\rangle=N^\ast\sigma$ and $\langle \ast\,dR,R\,dR\rangle=R^\ast\sigma$ 
(\cite{BFLPP02}, Proposition 10). Hence the lemma holds.  
\end{proof}
If $N^\ast\sigma=0$, then $N$ is not anti-holomorphic. 
Hence, if $f$ is super-conformal with $(df)_{-N}=(df)^R=0$, then 
$(1)$ $N^\ast\sigma=0$ and $R$ is anti-holomorphic or 
$(2)$ $R^\ast\sigma=0$ and $N$ is anti-holomorphic.

Wintgen \cite{Wintgen79} showed that $K+|K^\perp|\leq|\mathcal{H}|^2$ for any conformal map 
and $K+|K^\perp|=|\mathcal{H}|^2$ if and only if a conformal map is super-conformal. 
A super-conformal map is called a Wintgen ideal surface in \cite{PV08}. 
Chen \cite{ChenBY10} completely classified Wintgen ideal surfaces with $|K|=|K^\perp|$. 
We see that a Wintgen ideal surface with $|K|=|K^\perp|$ is 
a super-conformal map which is (i) minimal or (ii) $2K=2|K^\perp|=|\mathcal{H}|^2$. 
If a super-conformal map with left normal $N$ and right normal $R$ is minimal, then 
$\ast\,dN=N\,dN=-N\,dN$ or $\ast\,dR=R\,dR=-R\,dR$. 
Hence $N$ or $R$ is constant. 

We give a factorization of a minimal surface. 

Let $f\colon M\to\mathbb{H}$ be a minimal surface with $(df)_{-N}=(df)^R=0$. 
A minimal surface $g\colon M\to\mathbb{H}$ such that 
$dg=-\ast\,df$ is called a conjugate minimal surface of $f$. 
There exists a conjugate minimal surface of $f$ if and only if $\ast\,df$ is exact. 
A conjugate minimal surface $g$ shares the same left normal and the same right normal 
with the original minimal surface $f$. 
If there exists a conjugate minimal surface $g\colon M\to\mathbb{H}$, 
then the holomorphic map 
$f+ig\colon M\to\mathbb{C}\otimes \mathbb{H}\cong\mathbb{C}^4$ 
is called a holomorphic null curve. 

For a factorization of a minimal surface, we assume that $M$ is simply connected and induce 
a map $\mu$ as follows. 

If $N\colon M\to N(M)\subsetneq S^2\cong \mathbb{C}P^1$ is holomorphic, then $-N$ is anti-holomorphic by Corollary \ref{antihol}. 
By Lemma~\ref{psi}, there exists $a\in S^3$ such that 
$\psi:=-Na+ai$ does not vanish on $M$.
The map $\psi\lambda$ with $(d\lambda)_{-i}=0$ is super-conformal with left normal $-N$ by the factorization theorem for super-conformal maps.
Let $\lambda_0$ and $\lambda_1$ be holomorphic functions on $M$ and 
$\lambda:=\lambda_0+\lambda_1j$. 
Then, $(d\lambda)_{-i}=0$. 
Put $Q_{\lambda_0,\lambda_1}:=\{p\in M\,|\,(dN)_p\lambda(p)=0\}$. 
Because 
$(\psi\,d\lambda)_{N}=0$ and $(dN\,a\lambda)_N=0$, 
the equation $\psi\,d\lambda=dN\,a\lambda\mu$ defines 
a map $\mu \colon M\setminus Q_{\lambda_0,\lambda_1}\to\mathbb{H}$.

Because $N$ is holomorphic and $(d\lambda)_{-i}=0$, the set $Q_{\lambda_0,\lambda_1}$ is discrete.

\begin{theorem}\label{fctmin}
Let $N\colon M\to N(M)\subsetneq S^2\cong\mathbb{C}P^1$ be a holomorphic map and 
$a\in S^2$ such that $\psi:=-Na+ai$ does not vanish on $M$.
For complex functions $\lambda_0$ and $\lambda_1$ on $M$, set 
$Q_{\lambda_0,\lambda_1}:=\{p\in M\,|\,(dN)_p(\lambda_0(p)+\lambda_1(p)j)=0\}$.

$(1)$ If $\Phi:=f+ig\colon M\to\mathbb{C}\otimes\mathbb{H}$ is a holomorphic null curve 
and $f$ and $g$ are minimal surfaces with left normal $N$, 
then there exist holomorphic functions $\lambda_0$ and $\lambda_1$ on $M$, and 
$\mu\colon M\setminus Q_{\lambda_0,\lambda_1}\to\mathbb{H}$ with
$\psi\,d(\lambda_0+\lambda_1j)=dN\,a(\lambda_0+\lambda_1j)\mu$, 
such that 
$f=a(\lambda_0+\lambda_1j)(\mu-1)$ and
$g=-Na(\lambda_0+\lambda_1j)\mu+ai(\lambda_0+\lambda_1j)$ 
up to constant addition.

$(2)$ Let $\lambda_0$ and $\lambda_1$ be holomorphic functions on $M$. 
Define a map 
$\mu\colon M\setminus Q_{\lambda_0,\lambda_1}\to\mathbb{H}$ by
$\psi\,d(\lambda_0+\lambda_1j)=dN\,a(\lambda_0+\lambda_1j)\mu$. 
Then, the maps $f:=a(\lambda_0+\lambda_1j)(\mu-1)$ and
$g:=-Na(\lambda_0+\lambda_1j)\mu+ai(\lambda_0+\lambda_1j)$ 
are minimal surfaces with left normal $N$ and 
$\Phi:=f+ig\colon M\setminus Q_{\lambda_0,\lambda_1}\to\mathbb{H}$ is a holomorphic null curve.
\end{theorem}
\begin{proof}
$(1)$ 
We assume that $\Phi:=f+ig\colon M\to\mathbb{C}\otimes\mathbb{H}$ is a holomorphic null curve such that $f$ and $g\colon M\to\mathbb{H}$ are minimal surfaces with left normal $N$. 
We have $d(dN\,f))=-dN\wedge df=-(dN)^N\wedge (df)_N=0$. Then, 
the one-form $dN\,f$ on $M$ is exact. 
Hence, there exists a function $\Lambda\colon M\to\mathbb{H}$ such that 
$dN\,f=d\Lambda$. 
We define $\lambda\colon M\to\mathbb{H}$ by $\lambda:=\psi^{-1}\Lambda$. 
Then, 
\begin{gather*}
dN\,f=-dN\,a\lambda+\psi\,d\lambda=d(\psi\lambda). 
\end{gather*}
Because $(dN\,f)_{N}=0$ and $(dN\,a\lambda)_{N}=0$, we have 
$(\psi\,d\lambda)_{N}=0$. Then, 
\begin{gather*}
(\psi\,d\lambda)-N\ast(\psi\,d\lambda)
=\psi(d\lambda+i\ast\,d\lambda)=0. 
\end{gather*}
Hence, $(d\lambda)_{-i}=0$. 
Then, 
\begin{gather*}
dN\,f=dN\,a\lambda(\mu-1)
\end{gather*}
on $M\setminus Q_{\lambda_0,\lambda_1}$.
Then, $f=a\lambda(\mu-1)$. Because the left hand side is defined on $M$, the right hand side is 
extended to $M$. 
Then, 
\begin{gather*}
-\ast\,df=-N\,df=-d(Nf)+dN\,f=-d(Nf)-dN\,a\lambda+(-Na+ai)\,d\lambda\\
=d(-Na\lambda(\mu-1)+(-Na+ai)\lambda)
=d(-Na\lambda\mu+ai\lambda).
\end{gather*}
Hence $g= -Na\lambda\mu+ai\lambda$ up to an additive constant. 

$(2)$ 
We have 
\begin{gather*}
dN\,f=dN\,(a\lambda(\mu-1))=-dN\,a\lambda+dN\,a\lambda\mu\\
=-dN\,a\lambda+\psi\,d\lambda=d(\psi\lambda). 
\end{gather*}
Differentiating the above equation, we have 
\begin{gather*}
-dN\wedge df=(dN)^N\wedge (df)_{-N}=0. 
\end{gather*}
Hence $(df)_{-N}=0$. Then $f$ is a minimal surface with left normal $N$.
Then, 
\begin{gather*}
-\ast\,df=-N\,df=-d(Nf)+dN\,f=-d(Nf)-dN\,a\lambda+(-Na+ai)\,d\lambda\\
=d(-Na\lambda(\mu-1)+(-Na+ai)\lambda)
=d(-Na\lambda\mu+ai\lambda)
\end{gather*}
Hence, $g$ is a minimal surface with left normal $N$ and $\Phi:=f+ig$ is a holomorphic null curve. 
\end{proof}

The equation $f=a\lambda(\mu-1)$ is a factorization of a minimal surface which has 
a conjugate minimal surface $g$ and $\Phi:=f+gi$ is a holomorphic null curve. 
The arrangement of the zeros of $f$ is unclear because that of $\mu-1$ is unclear.
However, we have the following property. 
\begin{theorem}\label{zero}
Let $f=a\lambda(\mu-1) \colon M\to\mathbb{H}$ be a minimal surface factored by Lemma \ref{fctmin}. 
A point on $M$ is a branch point of a super-conformal map $\psi\lambda$ if and only if 
it is a zero of $f$ or a branch point of $N$. 
\end{theorem}
\begin{proof}
From the proof of Lemma \ref{fctmin}, we have $dN\,f=d(\psi\lambda)$. 
Hence the corollary holds. 
\end{proof}
\bibliographystyle{ijmart}

\begin{thebibliography}{99}
\bibitem{Ahlfors81}
L. V. Ahlfors,
\textit{M\"obius transformations in several dimensions},
Ordway Professorship Lectures in Mathematics, 
University of Minnesota, School of Mathematics, 
Minneapolis, 1981.
\bibitem{Boruvka28}
O. Bor\r{u}vka,  
\textit{Sur une classe de surfaces minima plong\'ees dans un espace \`a quatre dimensions \`a courbure constante}, 
Bulletin int. Acad. Tch\`{e}que Sci. \textbf{29} (1928), 256-277. 
\bibitem{BFLPP02}F. E. Burstall, D. Ferus, K. Leschke, F. Pedit and U. Pinkall, 
\textit{Conformal geometry of surfaces in ${\it S}^4$ and quaternions}, 
Lecture Notes in Mathematics \textbf{1772},
Springer-Verlag, 
Berlin, 
2002. 
\bibitem{Castro04}
I. Castro, 
\textit{Lagrangian surfaces with circular ellipse of curvature in complex space forms}, 
Math. Proc. Cambridge Philos. Soc.
\textbf{136} (2004), no. 1, 
239--245. 
\bibitem{ChenBY10}
B.-Y. Chen, 
\textit{Classification of Wintgen ideal surfaces in Euclidean 4-space with equal Gauss and normal curvatures},
Ann. Global Anal. Geom. 
\textbf{38} 
(2010), 
no. 2,
145--160. 
\bibitem{CO67}
S. Chern and R. Osserman, 
\textit{Complete minimal surfaces in euclidean {$n$}-space}, 
J. Analyse Math. \textbf{19}  
(1967), 
15--34. 
\bibitem{DT09}
M. Dajczer and R. Tojeiro, 
\textit{All superconformal surfaces in $\mathbb{R}^4$ in terms of minimal surfaces}, 
Math. Z. \textbf{261} (2009), no. 4, 869--890. 
\bibitem{FLPP01}
D. Ferus, K. Leschke, F. Pedit and U. Pinkall, 
\textit{Quaternionic holomorphic geometry: Pl\"ucker formula, Dirac eigenvalue
    estimates and energy estimates of harmonic $2$-tori}, Invent. Math. 
\textbf{146} (2001), no. 3, 507--593. 
\bibitem{Forster91}
O. Forster, 
\textit{Lectures on Riemann surfaces}, 
Graduate Texts in Mathematics \textbf{81}, 
Springer-Verlag, 
New York, 
1991. 
\bibitem{Friedrich97}
T. Friedrich, 
\textit{On superminimal surfaces},
Arch. Math. (Brno) 
\textbf{33} (1997), no. 1-2, 
41--56. 
\bibitem{GK06}
R. E. Greene and S. G. Krantz, 
\textit{Function theory of one complex variable},
Graduate Studies in Mathematics \textbf{40}, Third edition,  
American Mathematical Society, 
Providence, RI, 
2006. 
\bibitem{Kobayashi67}
S. Kobayashi, 
\textit{Intrinsic metrics on complex manifolds},
Bull. Amer. Math. Soc. 
\textbf{73} (1967), 347-349.
\bibitem{Moriya98}
K. Moriya, 
\textit{On a variety of algebraic minimal surfaces in {E}uclidean
              {$4$}-space}, 
Tokyo J. Math. 
\textbf{21}  
(1998), 
no. 1, 
121--134. 
\bibitem{Moriya08}
K. Moriya, 
\textit{The denominators of Lagrangian surfaces in complex Euclidean plane}, 
Ann. Global Anal. Geom. \textbf{34} (2008), no. 1, 1--20. 
\bibitem{Moriya09}
K. Moriya, 
\textit{Super-conformal surfaces associated with null complex holomorphic curves}, 
Bull. Lond. Math. Soc. \textbf{41} (2009), no. 2, 327--331. 
\bibitem{Moriya10}
K. Moriya, 
\textit{Quotients of quaternionic holomorphic sections},
Riemann surfaces, harmonic maps and visualization, 
OCAMI Stud. \textbf{3}, 
Osaka Munic. Univ. Press, Osaka, 
2010, 
pp. 197--201. 
\bibitem{PP98}
F. Pedit and 
U. Pinkall, 
\textit{Quaternionic analysis on Riemann surfaces and differential geometry}, 
Proceedings of the International Congress of Mathematicians, Vol. II (Berlin, 1998), 
Doc. Math. (1998), Extra Vol. II, 389--400 (electronic). 
\bibitem{PV08}
M. Petrovi{\'c}-Torga{\v{s}}ev and L. Verstraelen, 
\textit{On Deszcz symmetries of Wintgen ideal submanifolds},
Arch. Math. (Brno) 
\textbf{44} (2008), no. 1,
57--67. 
\bibitem{Rouxel89}
B. Rouxel,
\textit{Sur quelques propri\'et\'es conformes des surfaces de Bor\r{u}vka de $E^4$},
Czechoslovak Math. J. \textbf{39(114)} (1989), 
no. 4, 
604--613. 
\bibitem{Wintgen79}
P. Wintgen, 
\textit{Sur l'in\'egalit\'e de Chen-Willmore},
C. R. Acad. Sci. Paris S\'er. A-B \textbf{288} 
(1979),  no. 21, 
A993--A995. 
\bibitem{Wong52}
Y.-C. Wong, \textit{A new curvature theory for surfaces in a Euclidean $4$-space},
Comment. Math. Helv. \textbf{26} (1952), 
152--170. 
\bibitem{Wong46}
Y.-C. Wong, \textit{Contributions to the theory of surfaces in a 4-space of
              constant curvature}, 
Trans. Amer. Math. Soc.  
\textbf{59} (1946), 
467--507. 
\end{thebibliography}

\end{document}